\documentclass[11pt]{amsart}

\usepackage{amsmath}
\usepackage{amsfonts}
\usepackage{amssymb}
\usepackage{amsthm}
\usepackage{url}
\usepackage[usenames,dvipsnames]{color}
\usepackage[normalem]{ulem}
\usepackage{wasysym}
\usepackage{amscd}
\usepackage{setspace}
\usepackage{pslatex}
\usepackage{eucal}
\usepackage{enumerate}
\usepackage{hyperref}
\usepackage{paralist}

\theoremstyle{plain}
\newtheorem{thm}{Theorem}[section]
\newtheorem{lem}[thm]{Lemma}
\newtheorem{cor}[thm]{Corollary}
\newtheorem{prop}[thm]{Proposition}

\theoremstyle{remark}
\newtheorem{rem}[thm]{Remark}
\newtheorem{exmp}[thm]{Example}

\theoremstyle{definition}
\newtheorem{defn}[thm]{Definition}
\newtheorem{notation}[thm]{Notation}

\newtheorem{qstn}[thm]{Question}

\newcommand{\magma}{{\sc Magma}}

\allowdisplaybreaks

\title{Arithmetic dynamics on smooth cubic surfaces}
\author{Solomon Vishkautsan}
\address{Department of Mathematics,
Ben-Gurion University of the Negev,
P.O.B. 653 Beer-Sheva
8410501 Israel}
\email{wishcow@gmail.com}
\date{October 2013}

\begin{document}
\begin{abstract} 
We study dynamical systems induced by birational automorphisms on smooth cubic surfaces defined over a number field $K$. In particular we are interested in the product of non-commuting birational Geiser involutions of the cubic surface. We present results describing the sets of $K$ and $\bar{K}$-periodic points of the system, and give a necessary and sufficient condition for a dynamical local-global property called strong residual periodicity. Finally, we give a dynamical result relating to the Mordell--Weil problem on cubic surfaces.
\end{abstract}

\maketitle


\section{Introduction}

In this article we study arithmetic dynamics on smooth cubic surfaces over a number field $K$. The setup is quite simple: Take $X/K$ a smooth cubic surface over a number field $K$, and $f$ a birational automorphism of $X$, also defined over $K$. A dynamical system is induced by applying iterations of $f$ to points in $X(K)$ or $X(\bar{K})$ (where $\bar{K}$ is the algebraic closure of $K$). Such a dynamical system is an example of an \emph{arithmetic-geometric dynamical system} (described formally in Section \ref{section:dynamical-systems}). We are interested in two questions: What can be said about the  $K$- and $\bar{K}$-periodic points of $f$? What can be said about the interplay between global dynamics over $K$ and local dynamics when the system is reduced modulo $p$, for all but finitely many primes $p$ in $K$'s ring of integers?

In particular, we focus on dynamical systems induced by a simple type of birational automorphism on smooth cubic surfaces, defined by taking the composition of two Geiser involutions of the cubic surface (see Section \ref{section:cubic-surfaces}). Such automorphisms are examples of \emph{Halphen twists} (cf.\ Brown and Ryder\cite{article:brown-ryder2010}, Blanc and Cantat \cite{article:blanc-cantat2013}). 
By a theorem of Manin (\cite[Example 39.8.4]{book:manin1986}), the composition of two Geiser involutions is of infinite order in the group of birational automorphisms of the cubic surface, and has the nice property of preserving an elliptic fibration of the cubic surface (by \emph{preserve} we mean that every fiber is mapped to itself under the birational automorphism). 
Even for this simple type of birational automorphisms, the dynamical properties are not entirely trivial and deserve to be studied carefully.

A complication in studying the dynamics of a birational map $\varphi$ of an algebraic variety $X$ is the \emph{locus of indeterminacy} $\mathcal{Z}(\varphi)$, the set of points where the map $\varphi$ is not defined. Even worse is the fact that the set of points on the variety whose iterations under $\varphi$ land in $\mathcal{Z}(\varphi)$, which we denote by $\mathcal{Z}_\infty(\varphi)$, can \textit{a priori} be the set of all rational points of $X$ defined over $\bar{K}$. A recent result by Amerik \cite[Corollary 9] {article:amerik2011} states that this in fact cannot happen, but does not guarantee any periodic points lying outside of this set. We also need to define what we mean by periodic points in this setup, since for example for a birational involution $\varphi$ of the projective plane, the image of a point can be undefined for the first iteration of $\varphi$, but fixed under the second iteration. We do not wish to consider points of this type to be periodic, so in this article we will only consider a point to be periodic if it lies outside of $\mathcal{Z}_\infty(\varphi)$.

An arithmetic-geometric dynamical system $D$ (such as the one induced by a birational automorphism of a smooth cubic surface) can be reduced modulo $p$ for all but finitely many primes $p$, inducing \emph{residual dynamical systems} $D_p$ (see Section \ref{section:dynamical-systems}). In some systems, an interesting local-global behavior occurs when the system $D$ has no periodic points defined over $K$, but there exist periodic points of bounded period modulo all but finitely many primes $p$, where the bound on the periods is independent of $p$. We consider the more general case, when there exist periodic points of bounded period modulo all but finitely many primes $p$ that are not reductions modulo $p$ of any periodic points over $K$, with the bound on the periods independent of $p$. This property was first described  by Bandman, Grunewald and Kunyavski{\u\i} \cite[Section 6]{article:bandman-grunewald-kunyavskii2010}, and is called \emph{strong residual periodicity}. In the above-mentioned article one can find motivating examples. 

Our results in this article are as follows: Let $f$ denote a birational automorphism defined by taking the composition of two Geiser involutions on a smooth cubic surface $X$. We show that the $\bar{K}$-periodic points of $f$ (lying outside of $\mathcal{Z}_\infty(f)$) are Zariski-dense in $X(\bar{K})$ (Corollary \ref{cor:main-corollary-A}). The set of $K$-periodic points is contained in the union of finitely many fibers of the elliptic fibration preserved by $f$ (Corollary \ref{cor:cubic-surface-finite-union-of-fibers-contains-periodic}), and the number of these fibers is bounded by a number depending only on the degree of the extension $K/\mathbb{Q}$. We further show that if $K=\mathbb{Q}$ then the period of $\mathbb{Q}$-periodic points is bounded by $12$, and cannot equal $11$ (Theorem \ref{thm:main-theorem-A}). The fibers containing all periodic points can be found using a sequence of recursively defined polynomials that relate to division polynomials of elliptic curves. We define and use these polynomials to study local-global dynamics of the system and provide a necessary and sufficient condition for strong residual periodicity (Theorem \ref{thm:main-theorem-B}). Finally, we provide a result relating to the Mordell--Weil problem on cubic surfaces (see Section \ref{section:mordell-weil}): We prove that under mild conditions, the set of periodic points of $f$ is finitely generated by tangents and secants (Theorem \ref{thm:cubic-finitely-generated}). We also call the reader's attention to the very useful Lemma \ref{lem:group-translation}, which proves that group translations are never strongly residually periodic. 

Let us briefly describe the structure of the article: In Sections 2-4 we provide the notations and preliminaries required to prove our results. In Section 5 we describe the dynamics of a product of Geiser involutions. In Section 6 we present a method for counting periodic fibers of such birational automorphisms. In Section 7 we prove a necessary and sufficient condition for strong residual periodicity of a product of Geiser involutions. In Section 8 we discuss the Mordell--Weil problem on cubic surfaces. In Section 9 we provide examples illustrating the various results of the article. 

Acknowledgments: This article contains some of the results from the author's PhD thesis under the joint advisorship of Tatiana Bandman and Boris Kunyavski{\u\i} of Bar-Ilan University. Their tremendous help and advice during the writing of the dissertation and this article are greatly appreciated. The author also thanks (in chronological order) Michael Friedman, Igor Dolgachev and Serge Cantat for fruitful correspondence and discussions relating to this article and its results. The author's research was supported by the Israel Science Foundation, grants 657/09 and 1207/12, and by the Skirball Foundation via the Center for Advanced Studies in Mathematics at Ben-Gurion University of the Negev. The author gratefully thanks the referee for his helpful comments and corrections.

\section{Notations}

\begin{notation}\mbox{}
\begin{itemize}
\item $K$ is a number field, and $\bar{K}$ its algebraic closure.
\item $\mathcal{O}_K$ is the ring of integers of $K$.
\item $p$ is a prime ideal in $\mathcal{O}_K$.
\item $\mathcal{O}_p$ is the localization of $\mathcal{O}_K$ at the prime $p$.
\item $\mathfrak{m}_p$ is the maximal ideal in $\mathcal{O}_p$.
\item $\kappa_p$ is the residue field of the prime $p$ (i.e.\ $\mathcal{O}_p/\mathfrak{m}_p$).
\item $\mathcal{L}(x,y)$ is the projective line going through two distinct points $x,y\in\mathbb{P}^3$.
\item $\mathcal{P}(x,y,z)$ is the projective plane through noncollinear $x,y,z\in\mathbb{P}^3$.
\item $T_x(S)$ is the tangent plane at $x$ for a smooth projective surface $S\subseteq\mathbb{P}^3$.
\end{itemize}
\end{notation}

\begin{defn} \label{def:indeterminacy}
Given a rational map $\varphi:X\rightarrow{Y}$ between quasiprojective varieties $X,Y$, the \emph{domain} of $\varphi$ is the largest open subset of $X$ for which the restriction of $\varphi$ is a morphism. The complement of the domain is called the \emph{locus of indeterminacy} (or \emph{indeterminacy set}), and is denoted by $\mathcal{Z}(\varphi)$.
\end{defn}

\begin{notation} \label{def:extended-indeterminacy-set}
Let $\varphi$ be a dominant rational self-map $\varphi:X\rightarrow{X}$ of a projective variety $X\subset\mathbb{P}^N_k$.
For an integer $n\geq{1}$ we denote 
$$\mathcal{Z}_n(\varphi) = \bigcup_{i={0}}^{n-1}\varphi^{-i}(\mathcal{Z}(\varphi)).$$
We remark that $\mathcal{Z}_n(\varphi) \neq \mathcal{Z}(\varphi^n)$ (e.g., $\mathcal{Z}_n(\varphi)$ may be an infinite set for a rational map $\varphi$ from a smooth projective surface, but $\mathcal{Z}(\varphi^n)$ is finite, since the locus of indeterminacy of such a map has codimension $\geq{2}$).
We also denote 
$$\mathcal{Z}_\infty(\varphi) = \bigcup_{n=1}^{\infty}\mathcal{Z}_n(\varphi) = \bigcup_{i=0}^{\infty}\varphi^{-i}(\mathcal{Z}(\varphi)).$$ 
Thus, $\mathcal{Z}_\infty(\varphi)$ is the set of all points whose orbit intersects the locus of indeterminacy $\mathcal{Z}(\varphi)$ (the Zariski-closure of this set is called the \emph{extended indeterminacy set}, cf.\ Diller \cite[Definition 2.1]{article:diller1996}).
\end{notation}

\section{Preliminaries on arithmetic-geometric dynamical systems} \label{section:dynamical-systems}

In this section, we recall the definitions and properties of arithmetic dynamical systems. We follow Silverman \cite[Section 1]{article:silverman2008}, Hutz \cite[Section 2]{article:hutz2009} and Bandman, Grunewald and Kunyavski{\u\i} \cite[Section 6]{article:bandman-grunewald-kunyavskii2010} unless otherwise stated.

\begin{defn}
Let $K$ be a number field. A triple $D=(X, \varphi, F)$ is called an \emph{arithmetic-geometric dynamical system} over $K$ (or \emph{$K$-dynamical system} or \emph{AG dynamical system}) if:
\begin{itemize}
\item $X$ is an algebraic $K$-variety;
\item $\varphi:X\rightarrow{X}$ is a dominant $K$-endomorphism (or a dominant $K$-rational self-map; note that in this case we allow the function in the dynamical system to be partially defined);
\item $F\subset{X(K)}$ is a subset of rational points which we call the \emph{forbidden set} of the dynamical system $D$ (the forbidden set will include points the dynamical behavior of which we wish to ignore, cf.\ Remark \ref{rem:forbidden-set} below). 
\end{itemize} 
\end{defn}

The dynamics are induced by the components of the system $D$: we study iterations $\varphi^n(x)$ for points $x\in{X(\bar{K})}$. We are particularly interested in \emph{periodic points}, i.e.\ points $x\in{X(\bar{K})}\setminus{\mathcal{Z}_\infty(\varphi)}$ such that $\varphi^n(x)=x$ for some positive integer $n$. The minimal such $n$ is called the \emph{exact period} of $x$.

\begin{defn}
Let $\mathcal{O}_K$ be the ring of integers of a number field $K$. A triple $\mathcal{D}=(\mathcal{X}, \Phi,\mathcal{F})$ is called an \emph{$\mathcal{O}_K$-dynamical system} if:
\begin{itemize}
\item $\mathcal{X}$ is an $\mathcal{O}_K$-scheme of finite type;
\item $\Phi:\mathcal{X}\rightarrow{\mathcal{X}}$ is a dominant $\mathcal{O}_K$-endomorphism (or a dominant $\mathcal{O}_K$-rational self-map);
\item $\mathcal{F}\subset{\mathcal{X}(\mathcal{O}_K)}$ is the \emph{forbidden set} of the dynamical system $\mathcal{D}$ (cf.\ Remark \ref{rem:forbidden-set} below).
\end{itemize}
\end{defn}

\begin{defn}
We say that an $\mathcal{O}_K$-dynamical system $\mathcal{D} = (\mathcal{X}, \Phi, \mathcal{F})$ is an
\emph{integral model} of the $K$-dynamical system $D=(X, \varphi, F)$ if:
\begin{itemize}
\item $\mathcal{X}\times_{\mathcal{O}_K}K=X$ (this means $X$ is the generic fiber of $\mathcal{X}$); 
\item the restriction of $\Phi$ to the generic fiber of $\mathcal{X}$ coincides with $\varphi$;
\item  $\rho(\mathcal{F})=F,$ where $\rho\colon\mathcal{X}(\mathcal{O}_K)\to X(K)$ is the restriction to the generic fiber.
\end{itemize}
\end{defn}

\begin{defn}
Consider a $K$-AG dynamical system $D=(X, \varphi, F)$ and an integral
model $\mathcal{D} = (\mathcal{X}, \Phi, \mathcal{F}),$ as described in the previous section. Let $p$ be a prime of $\mathcal{O}_K$.	 Then:
\begin{itemize}
\item $X_p$, the special fiber of $\mathcal{X}$ at $p$, is called the \emph{reduction of $X$ modulo $p$}. We have $X_p=\mathcal{X}\times_{\mathcal{O}_K}\kappa_p$;
\item let $\rho_p\colon\mathcal{X} \to X_p$ be the reduction map (restriction to the special fiber). The image of $a\in\mathcal{X}(\mathcal{O}_K)$ under $\rho_p$ is the \emph{reduction modulo $p$} of $a$;
\item $\varphi_p\colon X_p\to X_p$, the restriction of $\Phi$ to the special fiber over $p$, is an endomorphism (or rational self-map) of $\kappa_{p}$-schemes. This is the \emph{reduction of $\varphi$ modulo $p$};
\item $F_p=\rho_p(\mathcal{F})\subset X_p(\kappa_p)$ is the reduction of the forbidden set $\mathcal{F}.$
\end{itemize}
We call the triple $D_p = (X_p, \varphi_p, F_p)$ the \emph{residual system} of $D$ modulo $p$.
\end{defn}


%

\begin{defn} Let $D=(X, \varphi, F)$ be a $K$-AG dynamical system such that $X$ is a smooth and proper $K$-variety and $\varphi$ is a dominant endomorphism. Let $\mathcal{D} = (\mathcal{X}, \Phi, \mathcal{F})$ be an integral model. Then for a place $p$ of $K$ we say $X$ has \emph{good reduction at a prime $p$} if $X_p$ is a smooth and proper scheme; we say that $\varphi$ has \emph{good reduction at a prime $p$} if $\varphi_p$ extends to a dominant $\kappa_p$-morphism. If both $X$ and $\varphi$ have good reductions modulo $p$, we say that $D$ has \emph{good reduction} modulo $p$.
\end{defn}


Let us recall some important facts about good reduction: A smooth projective variety $X$ defined over a number field $K$ has good reduction at all but finitely many primes of $\mathcal{O}_K$ (see Hindry and Silverman \cite[Proposition~A.9.1.6]{book:hindry-silverman2000}). A similar statement is true for a morphism of a projective variety defined over a number field $K$ (see Hutz \cite[Proposition~1]{article:hutz2012}). For a prime $p$ of good reduction of a dynamical system $D$, reduction modulo $p$ commutes nicely with a morphism, i.e.\ $\rho_p(\varphi^n(a))=\varphi_p^n(\rho_p(a))$ (see Hutz \cite[Theorem~7]{article:hutz2009}). Therefore we can discuss the reductions of orbits, etc.

\begin{defn} \label{def:residual-periodicity}
Let $D=(X,\varphi,F)$ be a $K$-AG dynamical system, and let $D_p=(X_p, \varphi_p, F_p)$ be the reduction of $D$ modulo $p$ with respect to some integral model.
Let $a\in X_p(\kappa_p)\setminus F_p $  be a periodic point of $\varphi_p.$ Let
$\ell_p(\varphi, a)$ be the orbit size of $a.$ Set
$\underline{\ell}_p:=\min\{\ell_p(\varphi,a)\}$ where the minimum is taken over all $a$. If there are no periodic points in $X_p(\kappa_p)\setminus F_p, $ we
set $\underline{\ell}_p=\infty.$ Let $M$ denote the collection of primes $p$
such that $\underline{\ell}_p=\infty.$  Let $N=\{\underline{\ell}_p\}_{p\not\in M}.$ We say that a $K$-dynamical system $D=(X,\varphi, F)$ or an $\mathcal{O}_K$-dynamical system $\mathcal{D} = (\mathcal{X}, \Phi, \mathcal{F})$ is
\emph{residually aperiodic} if the set $M$ is infinite, \emph{residually periodic} if $M$ is finite, and \emph{strongly residually
periodic} (SRP) if the sets $M$ and $N$ are both finite. We denote by \emph{SRP(n)} a dynamical system that is strongly residually periodic with minimal periods bounded by an integer $n$ for all but finitely many primes.
\end{defn}

\begin{rem} \label{rem:forbidden-set} \mbox{}
Usually we will take the forbidden set $F\subset{X(K)}$ to be the set of all periodic points in $X(K)$ (or the Zariski-closure of this set), so that SRP describes the situation where we have bounded residual periods that cannot be explained by periodic points in $X(K)$. In case $\varphi$ is rational, we include $\mathcal{Z}_\infty$ in $F$ so that we can exclude points with bad dynamics.
\end{rem}

Given a dynamical system $D=(X/K, \varphi, \emptyset)$ (for now we ignore the forbidden set), one can ask about the orbit size of a point $a\in{X(K)}$ when reduced modulo $p$. If $a\in{X(\mathcal{O}_K)}$ is periodic of exact period $n$, then the orbit size of $\rho_p(a)$ divides $n$ (cf.\ Hutz \cite[Theorem~1]{article:hutz2009}). If $a$ is of an infinite orbit, we would expect that when reducing the point modulo primes $p$, the reduced point $\rho_p(a)$ will have periods that grow together with the cardinality of $\kappa_p$. This is a direct corollary from a theorem of Silverman (see \cite[Theorem~2]{article:silverman2008}):
%
%
%
%


\begin{cor}[Corollary to Silverman's Theorem] \label{cor:silverman}
Let $D=(X,\varphi, F)$ be an AG dynamical system, and let $a \in X(K)\setminus{\mathcal{Z}_\infty}$ be a point of an infinite orbit. Then the residual periods of $a$ are unbounded over the primes in $\mathcal{O}_K$.
\end{cor}

From the corollary we see that strong residual periodicity cannot be explained by one global point of infinite orbit, and in fact we can deduce from Silverman's theorem that it cannot be explained by a finite set of points of infinite orbit in $X(K)$. The corollary allows us to prove the following simple yet useful lemma:



\begin{lem} \label{lem:group-translation}
Let $G/K$ be an algebraic group over a number field $K$, and let $D$ be the dynamical system induced by the group translation $\varphi_g(x)=gx,$ for some element $g\in{G(K)}$ of infinite order. Then $D$ is not SRP.
\end{lem}

\begin{proof}
The iterations of $\varphi_g$ are very simple: $\varphi_g^n(x)=g^nx.$ We see that 
the following are equivalent:
\begin{enumerate}[(a)]
\item There exists a periodic point $x\in{G(K)}$ of exact period $n$.
\item The element $g$ is of finite order $n$.
\item The element $g$ is a periodic point of $\varphi_g$ of exact period $n$: i.e.\ $\varphi_g^n(g)=g$ (and $n\geq{1}$ is the minimal positive integer satisfying this).
\item The map $\varphi_g$ is of finite order $n$ in the automorphism group of the underlying variety of $G/K$. 
\end{enumerate}

Now, if $g$ is of infinite order in $G/K$, then by the properties above it is of infinite orbit, and we can use the corollary to Silverman's Theorem~(Corollary~\ref{cor:silverman}) to see that the $\varphi_g$-periods of $g$ modulo primes $p$ are unbounded over the primes. This means that the minimal periods are unbounded (because the equivalent conditions above are also relevant for the reduced system $D_p$ when there is good reduction), so that $\varphi_g$ is not SRP. 
\end{proof}

\section{Preliminaries on cubic surfaces}\label{section:cubic-surfaces}

In order to study arithmetic dynamics on smooth cubic surfaces, we need to recall several classical geometric properties and theorems related to them. 

We briefly recall the group structure on absolutely irreducible cubic plane curves (cf.\ Manin \cite[Chapter I, Section 1, page 7]{book:manin1986} and Silverman \cite[Chapter III, Section 2]{book:silverman2009}). Let $C$ be an absolutely irreducible cubic curve in the projective plane $\mathbb{P}^2$, defined over a field $k$. Let $C_{ns}(k)$ denote its set of nonsingular rational points over $k$. Assuming $C_{ns}(k)\neq\emptyset$, we define a binary composition law
$\circ:C_{ns}(k)\times{C_{ns}(k)}\rightarrow{C_{ns}(k)},$
by setting $x\circ{y}$ for $x\neq{y}$ to be the third point of intersection of the line $L=\mathcal{L}(x,y)$ with the curve $C$. If $x=y$ then we take $L$ to be the tangent to $C$ at $x$. 
We can then turn $C_{ns}(k)$ into a group by choosing an element $u\in{C_{ns}(k)}$ and defining $xy:=u\circ(x\circ{y})$ for $x,y\in{C_{ns}(k)}$. With this multiplication, $C_{ns}(k)$ is an abelian group with unit $u$ . If $C$ is smooth, this gives the usual composition law on elliptic curves, and we denote it by $x+y$.

%
Similarly, we can define a composition law $\circ:S(\bar{k})\times{S(\bar{k})}\rightarrow{S(\bar{k})}$ on a smooth cubic surface $S\subset\mathbb{P}^3_k$ defined over a field $k$ (cf.\ Manin \cite[Chapter I, Section 1, page 7]{book:manin1986}). This composition is only partially defined. Given two distinct points $x,y\in{S(\bar{k})}$ such that $\mathcal{L}(x,y)$ does not lie on $S$, we define $x\circ{y}$ to be the third point of the intersection of the line $\mathcal{L}(x,y)$ with $S$ (this line has three points of intersection with $S$, by Bezout's theorem). We note that this operation is commutative but not necessarily associative.

\begin{defn} \label{def:cubic-good-point}
A point $x\in{S(\bar{k})}$ on a smooth cubic surface $S/k$ is a \emph{good point} if it does not lie on the union of the lines of $S\times_k\bar{k}$ (recall that there are $27$ lines over $\bar{k}$ on a smooth cubic surface, cf.\ Shafarevich \cite[Chapter IV, Section 2.5, Theorem]{book:shafarevich1994-vol1}). An unordered pair of distinct points $x, y \in S(\bar{k})$ is called a \emph{good pair} if the line containing these points is not tangent to $S\times_k\bar{k}$ and does not intersect the union of the lines on $S\times_k\bar{k}$ in $\mathbb{P}^3$ (cf.\ Manin \cite[Chapter V, Section 33.6]{book:manin1986}).
\end{defn}

It is clear from the definition that a pair of distinct points $x,y\in{S(\bar{k})}$ on a smooth cubic surface is a good pair if and only if $x\circ{y}$ is defined, the three points $x,y$ and $x\circ{y}$ are distinct and all three of them are good.

The \emph{Geiser involution} of a smooth cubic surface $S$ through a point $x\in{S(k)}$, is a map $t_x:S\rightarrow{S}$ sending each $y\in{S(\bar{k})}$ to $x\circ{y}$, when defined (cf.\ Brown and Ryder \cite[Section 2.2]{article:brown-ryder2010} and Corti, Pukhlikov and Reid \cite[Section 2.6]{article:corti-pukhlikov-reid2000}). We define $t_x$ for absolutely irreducible cubic curves in the same way. It is clear that $t_x$ is a birational involution (it is generally not defined at $x$ itself). A theorem of Manin~\cite[Theorems 33.7, 33.8]{book:manin1986} says that the Geiser involutions together with the Bertini involutions and the projective automorphisms generate $Bir(S)$ for a minimal smooth cubic surface $S$ defined over a perfect non-closed field. Some other useful properties of the Geiser involution are that $t_x(y)=t_y(x)$ for a good pair $x,y$, and its locus of indeterminacy is $\mathcal{Z}(t_x)=\{x\}$. Also, $t_x$ is an automorphism when restricted to $S\setminus C_x$, where $C_x = T_x(S)\cap{S}$.




\begin{thm} \label{thm:txty-on-cubic-curve}
Let $C\subset\mathbb{P}^2$ be an absolutely irreducible plane cubic curve defined over a field $k$. Then:
\begin{enumerate}[(a)]
\item The product of any two Geiser involutions $t_xt_y$ is a group translation: Given the choice of a group structure on the nonsingular points on the cubic curve, then for any nonsingular point $z\in{C_{ns}(\bar{k})}$ we get $t_xt_y(z) = (y - x) + z$ (or $t_xt_y(z) = x^{-1}yz$ if the group is multiplicative.
\item For any $x,y,z\in{C}$ we have $t_xt_yt_z=t_w$, where $w=y\circ(x\circ{z})$. 
\end{enumerate}
\end{thm}

\begin{proof}
See the proof of Theorem 2.1 in Manin \cite{book:manin1986}. 
\end{proof}

\begin{thm} \label{thm:manin-infinite-order}
Let $S$ be a smooth cubic surface over a perfect field $k$, and let $x,y\in{S(k)}$ be a good pair. Then the birational map $t_xt_y$ is of infinite order in $Bir(S)$.
\end{thm}

\begin{proof}
See Manin \cite[Example 39.8.4]{book:manin1986}.
\end{proof}


We recall some facts about hyperplane sections of a smooth surface $S$ in $\mathbb{P}^3$ defined over a perfect field $k$. If $H$ is a hyperplane in $\mathbb{P}^3$, then a point $x\in{}S\cap{H}$ is singular (on $S\cap{H}$) if and only if $H=T_x$ (here $S\cap{H}$ is viewed scheme-theoretically, since the intersection may not be reduced, cf.\ Beltrametti et al.\ \cite[Chapter 3, Section 1.8]{book:beltrametti-et-al2009}). For a smooth cubic surface $S$ in $\mathbb{P}^3$ we know that any hyperplane section will be one of the following (see Reid \cite[Chapter 7, Section 1, Proposition]{book:reid1988}): 
%
	%
	%
	%
\begin{inparaenum} [\itshape a\upshape)]
\item an absolutely irreducible smooth plane cubic curve;
\item a cuspidal plane cubic;
\item a nodal plane cubic;
\item an absolutely irreducible conic and a line;
\item three distinct lines.
\end{inparaenum}
By using these properties it is easy to prove the following list of statements about hyperplane sections of cubic surfaces:

\begin{prop} \label{prop:irreducible-cubic-fiber}\label{prop:plane-section-through-L_x_y-points-on-line-are-nonsingular}
Let $S\subset\mathbb{P}^3$ be a smooth cubic surface over a perfect field $k$.
\begin{enumerate}[(a)]
\item The point $x\in{S(\bar{k})}$ is a good point if and only if the curve $C_x=T_x(S)\cap{S}$ is absolutely irreducible.
\item Let $x,y$ be distinct good points on $S$. Then $C_x$ and $C_y$ do not have any common components.
\item Let $x,y$ be a good pair on $S$. Then any plane $H\subset\mathbb{P}^3$ passing through $x,y$ intersects $S$ in an absolutely irreducible cubic curve $C$, and the three (distinct) points $x,y,z$ in $\mathcal{L}(x,y)\cap{S}$ are nonsingular on $C$.
\end{enumerate}
\end{prop}

\section{Dynamics of a product of Geiser involutions} \label{section:dynamics-geiser}


In this section we study the global dynamics of a product of two Geiser involutions $t_xt_y$, where $x,y$ is a good pair on $S$. We will show that the dynamics of $t_xt_y$ are determined by its restrictions to the fibers of the elliptic fibration preserved by $t_xt_y$. Using Theorem \ref{thm:txty-on-cubic-curve} we see that $t_xt_y$ restricts to a group translation on the nonsingular points of each fiber, making the dynamics of $t_xt_y$ particularly easy to study. As a slight disclaimer, we mention that some of the proofs in this section are classical in nature, so no originality is claimed here (other than applying them to the dynamical setting). Our main results in this section are Proposition \ref{prop:periodic-point-on-fiber-txty}, characterizing the periodic fibers of exact period $n$, and Proposition \ref{prop:props-Z_n(t_xt_y)}, proving the existence of $\bar{K}$-periodic points lying outside of $\mathcal{Z}_\infty(t_xt_y)$.

\begin{prop}  \label{prop:cubic-plane-section-invariant} \label{prop:fiber-group-translation}
Let $S$ be a smooth cubic surface, and $x,y$ a good pair. Let $H$ be a hyperplane going through $x,y$. Let $C$ be the hyperplane section $H\cap{S}$ (absolutely irreducible by Proposition \ref{prop:irreducible-cubic-fiber}). Then $C$ is invariant under $t_xt_y$ and the restriction of $t_xt_y$ to $C$ is a group translation on $C_{ns}$.
\end{prop}

\begin{proof}
 \begin{inparaenum} [\itshape(a)]
That $C$ is invariant is clear from the definition of $t_x$ and $t_xt_y$ is a group translation by Theorem \ref{thm:txty-on-cubic-curve}.
\end{inparaenum}
\end{proof}

%
%
%

Let $S \subset \mathbb{P}^3$ be a smooth cubic surface over a perfect field $k$. An \emph{elliptic fibration} on $S$ is a rational map $\varphi: S\rightarrow{B}$ defined over $k$, such that the geometric generic fiber is birational to a curve of genus~$1$. For a field $k$ of characteristic~$0$ the base of an elliptic fibration must be of genus~$0$ and has a rational point, so it is isomorphic to $\mathbb{P}^1$(see Brown and Ryder \cite[Section 1]{article:brown-ryder2010} for the definition, and a proof of the last statement). Given a good pair $x,y$ on ${S}$, we can define an elliptic fibration by taking the pencil of planes through the line $L=\mathcal{L}(x,y)$: we choose two distinct planes $H_1 = \{f=0\}, H_2 = \{g=0\}$ passing through $L$ (where $f,g$ are linear forms); then the rational map $\varphi=(f,g)$ is an elliptic fibration. We call such a fibration the \emph{linear fibration} associated with the good pair $x,y$. The following is immediate from Proposition~\ref{prop:cubic-plane-section-invariant}.

\begin{prop} \label{prop:geiser-fibration-stabilized}
Let $S \subset \mathbb{P}^3$ be a smooth cubic surface. Given a good pair $x,y$ on ${S}$, the fibers of the linear fibration through $x,y$ are invariant under the birational automorphism $t_xt_y$.
\end{prop}

From Proposition \ref{prop:fiber-group-translation} and the proof of Lemma \ref{lem:group-translation}, we see that the only fibers containing periodic points are those for which $t_xt_y$ restricts to a group translation of finite order (aside from the singular points of the singular fibers, which we will show to be fixed points). Denote by $Fixed(\varphi)$ the set of fixed points of a rational map $\varphi$, then:

\begin{prop}  \label{prop:fixed-txty}
Let $x,y$ be a good pair on a smooth cubic surface $S$ over a perfect field $k$, then the set of fixed points of $t_xt_y$ is equal to $Fixed(t_x)\cap{}Fixed(t_y)$.  
\end{prop}

\begin{proof}
It is obvious that if $w\in{Fixed(t_x)\cap{}Fixed(t_y)}$, then $w\in{Fixed(t_xt_y)}$. In the other direction, suppose $t_xt_y(w)=w$. Assume $w\neq{x}$; then we can apply $t_x$ to both sides of the equation and get $t_x(w)=t_y(w)$ (since $w\not\in\{x\}=\mathcal{Z}(t_x)$). We know that $t_x(w)=t_w(x)$, so that we get $t_w(x)=t_w(y)$. Now, if $t_w(x)\neq{w}$, we can apply $t_w$ to both sides of the equation and get $x=y$, a contradiction. Therefore, still under the assumption of $w\neq{x}$, we get $t_w(x)=t_x(w)=w$ and also $t_y(w)=w$ as required. If $w=x$, then we have $t_xt_y(x)=x$, which implies $t_y(x)\in{C_x}$ since $t_x^{-1}(x)=C_x$. The points $x,y, t_y(x)$ are collinear, and since $x,t_y(x)\in{C_x}$, we get that $y\in{C_x}$ as well, which implies $t_y(x)=x$; but then $t_xt_y(x)=t_x(x)$ is indeterminate, a contradiction. 
%
\end{proof}

\begin{cor} \label{cor:txty-fixed-iff-singular}\label{cor:12-fixed-points}
Let $S,x,y$ be as in Proposition \ref{prop:fixed-txty}. Then:
\begin{enumerate}[(a)]
\item A point $w\in{S(\bar{k})}$ is a fixed point of $t_xt_y$ if and only if $w$ is a singular point of the curve $C=\mathcal{P}(x,y,w)\cap{S}$.
\item The map $t_xt_y$ has at most $12$ fixed points over $\bar{k}$.
\end{enumerate}
\end{cor}

\begin{proof} \mbox{}
\begin{enumerate}[(a)]
\item We saw in the proof of Proposition \ref{prop:fixed-txty} that $w$ is a fixed point of $t_xt_y$ if and only if $x,y\in{C_w}$. Therefore $C=C_w$, and $w$ is singular on $C_w$ (by Proposition \ref{prop:irreducible-cubic-fiber}).
\item The linear fibration $\pi$ associated with $x,y$ can be blown up at the three points $x,y,z$ in $\mathcal{L}(x,y)\cap{S}$ to give a rational elliptic surface, which only has at most $12$ singular fibers (see Miranda \cite[Section 1]{article:miranda1990}) (Note that no singularities on the fibers of $\pi$ are resolved by the blowup, since $x,y,z$ are smooth on all fibers by Proposition \ref{prop:plane-section-through-L_x_y-points-on-line-are-nonsingular}). Each of these singular fibers is associated with a fixed point of $t_xt_y$ by Corollary \ref{cor:txty-fixed-iff-singular}. \qedhere
\end{enumerate}
\end{proof}

We will say that a fiber $C$ in the linear fibration associated with a good pair $\{x,y\}$ is a \emph{periodic fiber} of period $n>0$ if $(t_xt_y)^n$ is the identity when restricted to the fiber $C$, and $n$ is the minimal positive integer satisfying this.

\begin{prop} \label{prop:periodic-point-on-fiber-txty}
Let $x,y$ be a good pair on a smooth cubic surface $S\subset\mathbb{P}^3$ defined over a perfect field $k$. Let $w\in{S(\bar{k})}\setminus\mathcal{Z}_\infty(t_xt_y)$ be noncollinear with $x,y$, and denote $C=\mathcal{P}(x,y,w)\cap{S}$. Then the following are equivalent:
\begin{enumerate}[(a)]
\item The point $w$ is a periodic point of exact period $n>1$ of $t_xt_y$.
\item The curve $C$ is $t_xt_y$-periodic of period $n$ (which is the same as saying $t_xt_y$ is of order $n$ in $Aut(C)$), and $w$ is a nonsingular point of $C$. 
\item The point $y$ is of order $n$ in the group $C_{ns}(k)$ with $x$ chosen to be the unit element (the point $y$ is nonsingular on $C$ by Proposition \ref{prop:plane-section-through-L_x_y-points-on-line-are-nonsingular}).
\end{enumerate}
\end{prop}

\begin{proof}
Let $w\in{S}$ be periodic of exact period $n>1$. By Proposition~\ref{prop:cubic-plane-section-invariant}, the birational automorphism $t_xt_y$ restricts to an automorphism of $C$. We choose $x$ to be the unit element of the group structure on $C_{ns}(k)$ (note that $w\in{C_{ns}(k)}$ since the period of $w$ is greater than $1$, and then by Corollary \ref{cor:txty-fixed-iff-singular} it is nonsingular), and get
\begin{equation} \label{equation:group-translation}
w=(t_xt_y)^{n}(w)=ny+w
\end{equation}
(see Theorem \ref{thm:txty-on-cubic-curve}), from which we get $ny=0$, meaning that the order of $y$ on the cubic curve is $n$ and that $(t_xt_y)^n = id$. If the order of $t_xt_y$ was less than $n$, we would get a contradiction to $n$ being the exact period of $w$. So we have proved that \textit{(a)} implies \textit{(b)} and \textit{(c)}. Similarly, one uses equation (\ref{equation:group-translation}) to prove the other implications.
\end{proof}

We can say more for the period $n=2$:

\begin{prop}\label{prop:txty-period-2}
Let $S,x,y,\varphi,w$ and $C$ be as in Proposition \ref{prop:periodic-point-on-fiber-txty}.
We restrict $\circ$ to the curve $C$, where it is fully defined. Then the following are equivalent:
\begin{enumerate}[(a)]
\item \label{prop:txty-period-2:item-a} The point $w\in{S(\bar{k})}$ is a periodic point of exact period $n=2$ of $t_xt_y$.
\item \label{prop:txty-period-2:item-d} $x\circ{x} = y\circ{y}$.
\item \label{prop:txty-period-2:item-e} $x\circ{x}\in{C_x\cap{C_y}}$.
\end{enumerate}
\end{prop}

\begin{proof}
The statements \textit{(\ref{prop:txty-period-2:item-d})} and \textit{(\ref{prop:txty-period-2:item-e})} are equivalent, because $x\circ{x}\in{C_y}$ means that the line through $y$ and $x\circ{x}$ has a double point at $y$, and as this line is contained in $\mathcal{P}(x,y,w)$, it must mean that this is the tangent line to $C$ at $y$, but then by definition $x\circ{x}=y\circ{y}$. The statements \textit{(\ref{prop:txty-period-2:item-a})} and \textit{(\ref{prop:txty-period-2:item-d})} are equivalent since
$$(t_xt_y)^2(w)=w \iff 2y+w=w \iff 2y=0 \iff x\circ(y\circ{y})=x,$$ and we can apply $x$ to both sides of the last equation. 
\end{proof}

\begin{cor} \label{cor:curves-of-period-2}
Let $S,x,y, \varphi$ be as in Proposition \ref{prop:periodic-point-on-fiber-txty}; then there are three periodic fibers of period $2$ defined over $\bar{k}$, counted with multiplicity, in the linear fibration of $S$ through the points $x,y$ (and these contain all periodic points of period $2$).
\end{cor}

\begin{proof}
We show that the period $2$ fibers of $\varphi$ are determined by the intersection of the line $L=T_x(S)\cap{T_y(S)}$ with $S$. The line $L$ cannot lie on $S$, since otherwise it lies on $T_x(S)$ and therefore is contained in $C_x = T_x(S)\cap{S}$; but $C_x$ is absolutely irreducible by Proposition \ref{prop:irreducible-cubic-fiber}. Therefore the line $L$ intersects $S$ at three points (counted with multiplicity). None of these three points are collinear with $x,y$ (otherwise $\mathcal{L}(x,y)$ has a double point at both $y$ and $x$). Each point $w\in{S\cap{L}}$ then determines a fiber of the linear fibration, and on this fiber we get $w=x\circ{x}=y\circ{y}$, since $\mathcal{L}(w,x)\subset{T_x(S)}$ and $\mathcal{L}(w,y)\subset{T_y(S)}$. The result then follows from Proposition \ref{prop:txty-period-2}.
\end{proof}



\begin{prop} \label{prop:props-Z_n(t_xt_y)}
Let $S$ be a smooth cubic surface over a perfect field $k$. Let $x,y$ be a good pair on $S$, and let $C$ be a fiber of the linear fibration through $x,y$. Then: 
\begin{enumerate}[(a)]
\item \label{prop:props-Z_n(t_xt_y):item-a} For any positive integer $n$, the set $C(\bar{k})\cap{\mathcal{Z}_n(t_xt_y)}$ is finite (see Notation \ref{def:extended-indeterminacy-set}).
\item \label{prop:props-Z_n(t_xt_y):item-b} If $C$ is periodic of period $n$, then the set $C(\bar{k})\cap{\mathcal{Z}_\infty(t_xt_y)}$ is finite, and this implies the existence of $\bar{k}$-periodic points outside of $\mathcal{Z}_\infty(t_xt_y)$.
\end{enumerate}
\end{prop}

\begin{proof} \mbox{}
To make notations simpler, we identify all algebraic sets with their underlying set of $\bar{k}$-points. 
Denote $f=t_xt_y$, and let $z$ be the third point in $\mathcal{L}(x,y)\cap{S}$.
We prove the proposition by induction on $n$. For $n=1$ we have $\mathcal{Z}_1(f)=\mathcal{Z}(f)=\{y,z\}$, so the assertion is true. 
Let $n>1$, and assume that the statement is true for any $m<n$. For $n\geq{2}$ we have
$\mathcal{Z}_n(f)=\mathcal{Z}(f)\cup f^{-1}\left[\mathcal{Z}_{n-1}(f)\right],$
so that 
$$C\cap\mathcal{Z}_n(f) = (C\cap\mathcal{Z}(f)) \cup (C \cap f^{-1}\left[\mathcal{Z}_{n-1}(f)\right]).$$
The first set in the union is finite, so it remains to prove that $C \cap f^{-1}\left[\mathcal{Z}_{n-1}(f)\right]$ is finite. It is easy to see that $C \subset f^{-1}\left[C\right]$, so that
$$C \cap f^{-1}\left[\mathcal{Z}_{n-1}(f)\right] \subset C\cap f^{-1}\left[C\cap \mathcal{Z}_{n-1}(f)\right].$$
The set $C\cap \mathcal{Z}_{n-1}(f)$ is finite by the induction hypothesis, so
$$C\cap \mathcal{Z}_{n-1}(f) \subseteq \{x,y,z,A_1,...,A_K\},$$
where $A_1,...,A_K$ are points in $C(\bar{k})\setminus\{x,y,z\}$. The inverse image of $\{y,A_1,\ldots,A_K\}$ under $f$ is finite (as can be checked readily from the definition of $t_x$ and $t_y$), so it remains to show that $C\cap{f^{-1}[\{x,z\}]}$ is finite.  Now $f^{-1}[\{z\}] = C_y$, which is an irreducible hyperplane section singular at $y$ by Proposition \ref{prop:irreducible-cubic-fiber}, and therefore has no common components with $C$ (the curve $C$ is irreducible and cannot have a singularity at $x,y,z$ by Proposition \ref{prop:plane-section-through-L_x_y-points-on-line-are-nonsingular}), so that their intersection is finite. 
Finally, $f^{-1}[\{x\}] = t_y^{-1}(C_x)$. As before $C \subset t_y^{-1}\left[C\right]$, so that $$C\cap t_y^{-1}\left[C_x\right] \subset t_y^{-1}\left[C\cap C_x\right].$$ The set $C\cap{C_x}$ is finite, so $C\cap{C_x} \subseteq \{x, B_1,...,B_M\},$
where $B_1,...,B_M$ are points in $C(\bar{k})\setminus\{x,y,z\}$ (The curve $C_x$ does not contain $y$ and $z$, since otherwise $x$ is a double point of $\mathcal{L}(x,y)\cap{S}$, which is impossible, since $x,y$ and $z$ are distinct). The inverse image of $\{x, B_1,...,B_M\}$ under $t_y$ is finite. This proves \textit{(\ref{prop:props-Z_n(t_xt_y):item-a})}.

To prove \textit{(\ref{prop:props-Z_n(t_xt_y):item-b})}, we note that any point of $C$ not in $\mathcal{Z}_n(t_xt_y)$ must be periodic of period $n$, and therefore cannot lie in $\mathcal{Z}_\infty(t_xt_y)$; but $C\cap\mathcal{Z}_n(t_xt_y)$ is finite by \textit{(\ref{prop:props-Z_n(t_xt_y):item-a})}, so that $C\cap\mathcal{Z}_\infty(t_xt_y)$ is finite as well.
\end{proof}



\section{Division polynomials associated with linear fibration}

We have seen in Section \ref{section:dynamics-geiser} that finding periodic points of $t_xt_y$ for a good pair $x,y$ on a smooth cubic surface $S$ defined over a number field $K$, is equivalent to studying periodicity on the fibers of the linear fibration defined by the points $x,y$ (Proposition~\ref{prop:geiser-fibration-stabilized}). We have also seen that for a fiber to have periodic points of exact period $n$, the fiber itself must be periodic of period $n$, and this is equivalent to $y$ being of order $n$ in the group structure induced by choosing $x$ as the unit element (Proposition~\ref{prop:periodic-point-on-fiber-txty}). We want to find all the fibers of the linear fibration that are periodic of finite period. In order to do so we employ \textit{division polynomials} of elliptic curves. We first recall the definition and basic properties of these polynomials.

\begin{defn} \label{def:division-polynomials}
Given an elliptic curve $E$ in Weierstrass form $y^2 = x^3 + Ax+B,$ we associate to it the \emph{division polynomials} $\psi_n, n\geq{0}$ in $\mathbb{Z}[x,y,A,B]$:
\begin{eqnarray*} 
\psi_0 &=& 0 , \qquad \psi_1 = 1, \qquad \psi_2 = 2y, \\
\psi_3 &=& 3x^4 +6Ax^2 +12Bx - A^2 ,\\
\psi_4 &=& 4y(x^6+5Ax^4+20Bx^3 - 5A^2x^2 - 4ABx - 8B^2 - A^3) ,\\
\psi_{2m+1} &=& \psi_{m+2}\psi_m^3 - \psi_{m-1}\psi_m^3 + 1, \quad \text{ for } m \geq 2 ,\\
\psi_{2m} &=& (2y)^{-1}\psi_m(\psi_{m+2}\psi_{m-1}^2 - \psi_{m-2}\psi_{m+1}^2), \quad \text{ for } m \geq 3
\end{eqnarray*}
\end{defn}

\begin{thm} Let $E$ be an elliptic curve defined over a number field $K$. Then the division polynomials have the following properties ($E[n]$ is the set of points in $E(\bar{K})$ with order dividing $n$):
\begin{enumerate}[(a)]
\item $\psi_{2n+1}, y^{-1}\psi_{2n}$ are polynomials in $\mathbb{Z}[x,A,B]$;
\item the roots of $\psi_{2n+1}$ are the $x$-coordinates of the points in $E[2n+1] \setminus \{\mathcal{O}\}$;
\item the roots of $y^{-1}\psi_{2n}$ are the $x$-coordinates of the points in $E[2n] \setminus E[2]$.
\end{enumerate}
\end{thm}
\begin{proof}
See Washington \cite[Chapter 3, Section 2]{book:washington2008}.
\end{proof}

\begin{rem}
To make the notations easier, we replace $\psi_n$ for even $n$ with $y^{-1}\psi_n$.
\end{rem}

\begin{prop} \label{prop:periodic-fibers-using-division-polynomials}
Let $S$ be a smooth cubic surface defined over a number field $K$, let $x,y\in{S(K)}$ be a good pair, and denote by $\pi:S\rightarrow\mathbb{P}^1$ the linear fibration associated with the good pair $x,y$.
There exist polynomials $\gamma_n(t)$ in $K[t]$, for $n\geq{3}$, whose roots in $\bar{K}$ lying outside a finite set $\mathcal{B}\subset\mathbb{P}^1$ correspond to fibers of $\pi$ that are $t_xt_y$-periodic of period $\geq{3}$ and dividing $n$.
\end{prop}

\begin{proof} 

We can ensure that the fiber at infinity is non-periodic and nonsingular. Outside this fiber $\pi$ induces a cubic pencil with parameter $t$, whose generic fiber is a smooth cubic curve $C$ in the projective plane $\mathbb{P}^2$ over the function field $K(t)$. The points $x,y$ induce two rational points (which we still denote by $x,y$) on the cubic curve $C$. We choose $x$ to be the unit element of the cubic curve $C$, which induces an elliptic curve group structure on $C$. We note that it is impossible for $y$ on $C$ to be of finite order when $x$ is chosen as unit element, since otherwise $t_xt_y$ is of finite order in $Bir(X)$, contradicting Theorem~\ref{thm:manin-infinite-order}.

We now use a Weierstrass transformation (see Shioda \cite[Section 2]{article:shioda1995}) to map our elliptic curve $(C, x)$ to an isomorphic elliptic curve in Weierstrass form over $K(t)$ $$E: v^2=u^3+A(t)u+B(t).$$ We denote the Weierstrass transformation by $\omega: (C, x) \rightarrow (E,\mathcal{O}).$ The point $y$ is mapped to a rational point on $E$, and we still denote this point by $y$. We can ensure that the Weierstrass form $E$ has coefficients in $\mathcal{O}_K[t]$: This is done by transforming $E$ to a \emph{minimal Weierstrass form} (cf.\ Silverman \cite[Chapter~VII, Section 1]{book:silverman2009}), i.e.\ an isomorphic copy of $E$ such that the valuation of $A$ and $B$ is nonnegative and minimal (in its isomorphism class) at all places of $K(t)$; then it is possible to get rid of the denominators so that all the coefficients of $A(t)$ and $B(t)$ are in $\mathcal{O}_K$.

There exists an open subset $U\subseteq\mathbb{P}^1$ such that for each $t\in{U}$, the specialization of the Weierstrass transformation $\omega$ to the fiber over a specific $t$ will remain an isomorphism of irreducible cubic curves. We include the complement of $U$ in $\mathcal{B}$. 

We now have an infinite sequence of division polynomials $\psi_n\in{\mathcal{O}_K[t][u]}, n\geq{3},$ of the elliptic curve $E/K(t)$. We evaluate these polynomials at the $u$-coordinate of the point $y$, and get an element $\tilde{\gamma}_n(t)$ in $K(t)$. The denominator of $\tilde{\gamma}_n$ depends only on $y\in{E}$; thus there are finitely many values of $t$ at which $\tilde{\gamma}_n(t)$ might have poles, which we include in $\mathcal{B}$.  We now define $\gamma_n$ to be the numerator of $\tilde{\gamma}_n$, for $n\geq{3}$.

It is clear that the fibers of $\pi$ lying over the roots of $\gamma_n(t)$ not in $\mathcal{B}$, are  fibers where $y$ is of finite order $\neq{2}$ and dividing $n$. Thus by Proposition \ref{prop:periodic-point-on-fiber-txty} these fibers are periodic of period $\neq{2}$ and dividing $n$.
\end{proof}

\begin{cor} \label{cor:periodic-fibers-using-division-polynomials}
Let $S/K,x,y$ and $\pi$ be as in Proposition \ref{prop:periodic-fibers-using-division-polynomials}.
There exist polynomials $\Psi_n(t)$ in $K[t]$, for $n\geq{3}$, whose roots in $\bar{K}$ correspond to all the fibers of $\pi$ that are periodic under the birational map $t_xt_y$, of period $\neq{2}$ and dividing $n$.
\end{cor}

\begin{proof}
We only need to check the order of $y$ on the finite number of fibers over the points in $\mathcal{B}$. If a fiber over a point $a\in\mathcal{B}$ is non-periodic, then we factor out $a$ as a root from the polynomials $\gamma_n(t)$. If it is periodic, then we need to make sure it appears as a root of the polynomials $\gamma_n(t)$. It remains to prove that the resulting polynomials have coefficients in $K$. This is true because if we have a point $a$ on whose corresponding fiber we have $y$ of finite order $n$, then the order of $y$ on the fibers over the conjugates of $a$ will be the same, and therefore they will be roots of the same polynomials.
\end{proof}

\begin{cor} \label{cor:main-corollary-A}
Under the assumptions of the last corollary, the $\bar{K}$-periodic points of $t_xt_y$ are Zariski-dense in $S(\bar{K})$.
\end{cor}

\begin{proof}
By Corollary \ref{cor:periodic-fibers-using-division-polynomials}, there are infinitely many periodic fibers of $t_xt_y$ over $\bar{K}$ (the polynomials $\Psi_p, \Psi_q$ have no common roots for distinct primes $p$ and $q$). By Proposition \ref{prop:props-Z_n(t_xt_y)}, these fibers must contain $\bar{K}$-rational points lying outside of $\mathcal{Z}_\infty(t_xt_y)$.
\end{proof}

%
%
%

\begin{defn}
We call the polynomials $\Psi_n, n\geq{3}$ in the last corollary the \emph{division polynomials} of the smooth cubic surface $S$ with respect to the good pair $x,y$ (this definition is nonstandard).
We now define the 1st and 2nd division polynomials: 
\begin{enumerate}
\item $\Psi_1(t)$ is defined to be the discriminant of Weierstrass form $E/K(t)$ from the proof of Proposition~\ref{prop:periodic-fibers-using-division-polynomials}.
\item  $\Psi_2(t)$ is defined to be the numerator of the $v$ coordinate of the point $y$ on $E/K(t)$. 
\end{enumerate}
\end{defn}

\begin{prop} \label{prop:psi1-psi2}
Let $S$ be a smooth cubic surface defined over a number field $K$, and let $x,y\in{S(K)}$ be a good pair.
\begin{enumerate}[(a)]
\item The roots of $\Psi_1$ in $\bar{K}$ correspond to the fibers of $\pi$ that contain fixed points of $t_xt_y$. This polynomial is of degree at most $12$.
\item The roots of $\Psi_2$ in $\bar{K}$ correspond to the fibers of $\pi$ that are periodic of period $2$. This polynomial is of degree at most $3$.
\end{enumerate}
\end{prop}

\begin{proof} \mbox{}
\begin{enumerate}[(a)]
\item The discriminant of a cubic curve in Weierstrass form is $0$ if and only if the curve is singular. A point $w$ on the surface $S$ is fixed under $t_xt_y$ if and only if the fiber $C=\mathcal{P}(x,y,w) \cap S$ is singular, by Corollary \ref{cor:txty-fixed-iff-singular}. There are at most $12$ fixed points of $t_xt_y$ on the surface $S$, by Corollary \ref{cor:12-fixed-points}, which explains the degree.
\item A point on an elliptic curve in Weierstrass form is of order $2$ if and only if its $v$ coordinate is $0$. By Proposition \ref{cor:curves-of-period-2} there are at most $3$ curves of period $2$.  \qedhere
\end{enumerate}
\end{proof}

\begin{notation} \label{notation:dynatomic-polynomials} Let $S$ be a smooth cubic surface defined over a number field $K$, let $x,y\in{S(K)}$ be a good pair and let $\Psi_n$ be the division polynomials associated with $x,y$. 
For $n\geq{1}$ define the polynomials
$\Phi_n(t)=\Psi_n(t)\prod\limits_{{d|n, d>2}}\Psi_d(t)^{-1}$.
The roots of $\Phi_n(t), n\geq{2},$ in $\bar{K}$ correspond to periodic fibers of $\pi$ of \emph{exact} period $n$.
\end{notation}

\begin{prop} \label{prop:bounded-number-of-periodic-fibers}
Let $S$ be a smooth cubic surface defined over a number field $K$ and let $x,y\in{S(K)}$ be a good pair.  \mbox{}
\begin{enumerate}[(a)]
\item For the case $K=\mathbb{Q}$, the polynomials $\Phi_n(t)$ do not have roots in $\mathbb{Q}$ for $n>12$ or $n=11$.
\item For a general number field $K$ there exists a positive integer $N$ such that $\Phi_n(t)$ does not have roots in $K$ for $n>N$. The bound $N$ depends only on the degree of the extension $K/\mathbb{Q}$.
\end{enumerate}
\end{prop}

\begin{proof} \mbox{}
\begin{enumerate}[(a)]
\item Mazur's torsion theorem \cite[Theorem 8]{article:mazur1978}
lists the possible $\mathbb{Q}$-rational torsion subgroups of elliptic curves over $\mathbb{Q}$. The maximal order of an element of finite order is bounded by $12$. If $\Phi_n(t)$ has a root in $\mathbb{Q}$ for $n>12$, then the fiber over this root is an elliptic curve over $\mathbb{Q}$ with a rational point of order $n$, a contradiction.
\item Merel's torsion theorem (see Merel \cite[Corollary]{article:merel1996}), guarantees a bound on the order of the $K$-rational torsion subgroup that depends only on the degree of the field extension. An explicit bound on the order of the torsion subgroup can be found in Parent  \cite[Corollary 1.8]{article:parent1999} (this bound is exponential in the degree of the extension of $K/\mathbb{Q}$). \qedhere
\end{enumerate}
\end{proof}

\begin{cor} \label{cor:cubic-surface-finite-union-of-fibers-contains-periodic}
Let $S$ be a smooth cubic surface defined over a number field $K$ and let $x,y\in{S(K)}$ be a good pair. The set of all periodic points of $t_xt_y$ in $S(K)$ is contained in a finite number of fibers of $\pi$.
\end{cor}

\begin{proof}
All periodic points must be contained in the fibers of the linear fibration lying over the $K$-roots of the polynomials $\Psi_n$. Therefore, by Proposition~\ref{prop:bounded-number-of-periodic-fibers}, there are only finitely many fibers that can contain $K$-periodic points. 
\end{proof}

\begin{thm}\label{thm:main-theorem-A}
Under the assumptions of Proposition \ref{prop:bounded-number-of-periodic-fibers}, for $K=\mathbb{Q}$ the birational automorphism $t_xt_y$ can only have $\mathbb{Q}$-periodic points of exact period $1,...,10,12$. For a general number field $K$, the birational automorphism $t_xt_y$ can only have periodic points of exact period bounded by a number depending only on the degree of the extension $K/\mathbb{Q}$.
\end{thm}

The division polynomials of the linear fibration give us an effective method for ``finding" all $K$-periodic points of a birational automorphism $t_xt_y$: 

\begin{enumerate}[(1)]
\item First we find the $K$-roots of all division polynomials up to the bounds described in the proof of Proposition \ref{prop:bounded-number-of-periodic-fibers}.
\item For the nonsingular periodic fibers, we can then use algorithms to find the generators of the Mordell-Weil of the elliptic curve (cf.\ Cremona \cite[Chapter 3, Section 5]{book:cremona1997}).
\item We can parametrize the periodic points on the singular periodic fibers.
\item For each periodic fiber, we can find all the points in $\mathcal{Z}_\infty(t_xt_y)$ in a finite number of steps (see Proposition \ref{prop:props-Z_n(t_xt_y)}).
\end{enumerate}

\section{SRP on cubic surface} \label{section:SRP-cubic-surface}

We are now ready to prove some results about strong residual periodicity of the AG dynamical system of type $D=(S,t_xt_y,F)$ on a smooth cubic surface $S$ defined over a number field $K$, whose global dynamics were described in the previous two sections. We take the forbidden set $F$ to be union of the Zariski-closure of the set of all periodic points in $X(K)$ and the set of $K$-rational points in $\mathcal{Z}_\infty(t_xt_y)$. In this section we prove two sufficient conditions for SRP, and then we prove that when put together, they are necessary and sufficient.

\begin{prop} \label{prop:srp-on-cubic-surface-cond1}
Let $S$ be a smooth cubic surface defined over a number field $K$ and let $x,y\in{S(K)}$ be a good pair.
If there exists a fiber $C$, defined over $K$, of the linear fibration through the good pair $x,y\in{S(K)}$ such that the set of $K$-rational points $C(K)$ is finite, then $D=(S,\varphi=t_xt_y,F)$ is strongly residually periodic.
\end{prop}

\begin{proof}
The curve $C$ must be smooth, since absolutely irreducible singular cubic curves with a rational point are rational, and therefore have infinitely many rational points. The point $y$ must be of finite order in the group induced by choosing $x$ as the unit element on $C$. All points in $C(K)$, outside of $\mathcal{Z}_\infty(t_xt_y)$, are $t_xt_y$-periodic with period equal to the order of $y$. Denote by $N$ the order of $y$ in $C(K)$. To prove the proposition we will show that for primes $p$ for whom the cardinality of $\kappa_p$ is large enough, there exists a point $w\in{C_p(\kappa_p)}$ that satisfies:
\begin{enumerate}[(1)]
\item The point $w$ is not a reduction of any periodic point in the Zariski-closure of the subset of periodic points in $S(K)$.
\item The point $w$ is not a reduction of any point in $\mathcal{Z}_\infty(\varphi)$.
\item The point $w$ is not in $\mathcal{Z}_\infty(\varphi_p)$ (in particular $w$ is defined under $\varphi_p^N$, and is therefore a periodic point of period at most $N$).
\end{enumerate}

We restrict ourselves to primes $p$ where the system $D$ has good reduction. In particular the elliptic curve $E=(C, x)$ is reduced to an elliptic curve $E_p=(C_p, x_p)$. We can then use Hasse's theorem on elliptic curves (see Silverman \cite[Thm V.1.1]{book:silverman2009}) to guarantee that the number of rational points in $C_p(\kappa_p)$ is as large as we desire. We can thus guarantee that for all but finitely many primes $p$ there exists a point $w\in{C_p(\kappa_p)}$ such that $w$ is not a reduction of any point in $C(K)$. The Zariski-closure of the subset of periodic points in $S(K)$ is contained in a finite number of fibers of the linear fibration through $x,y$, so that it is enough to choose $p$ such that $C_p$ intersects the reductions of the other fibers only at the points $\{x,y,z\}=\mathcal{L}(x,y)\cap{S}$ (this is true for any prime $p$ such that $\kappa_p$ is large enough). This proves $(1)$.

We proved in Proposition \ref{prop:props-Z_n(t_xt_y)} that the intersection of $C(\bar{K})$ with $\mathcal{Z}_\infty(\varphi)$ is finite. We use this and the same reasoning as in the previous paragraph to prove $(2)$.

The set $\mathcal{Z}_\infty(\varphi_p)$ is the set of points whose orbit intersects $\mathcal{Z}(\varphi_p)$. Now, $\mathcal{Z}(\varphi_p)\subseteq\{y_p = \rho_p(y), z_p = \rho_p(z)\}$, so that any point $v\in{C_p(\kappa_p)}$ not in $\mathcal{Z}(\varphi_p)$ that lies in $\mathcal{Z}_\infty(\varphi_p)$, satisfies $ny_p+v=y_p$ or $ny_p+v=z_p$, for some integer $1\leq{n}\leq{N}$. Rewriting this, we get $v=(1-n)y_p$ or $v=z_p-ny_p$ , for some integer $1\leq{n}\leq{N}$. In other words, there are at most $2N$ points in $C_p(\kappa_p)\cap\mathcal{Z}_\infty(\varphi_p)$. The bound $2N$ is independent of the prime $p$, and therefore for primes $p$ such that $\kappa_p$ is large enough, we can find points in ${C_p(\kappa_p)}$ that are not in $\mathcal{Z}_\infty(\varphi_p)$, so that $(3)$ is proved.
\end{proof}

\begin{rem} \mbox{}
\begin{inparaenum}[\itshape(i\upshape)]
\item  The order of $y$ on $C(K)$ is bounded by Theorem \ref{thm:main-theorem-A}.
\item See Example \ref{ex:srp-cubic} for a dynamical system satisfying the conditions of Proposition \ref{prop:srp-on-cubic-surface-cond1}.
\end{inparaenum}
\end{rem}

\begin{prop} \label{prop:cubic-theta-modulo}
Let $S$ be a smooth cubic surface defined over a number field $K$ and let $x,y\in{S(K)}$ be a good pair.
The minimal periods of the residual systems of $D=(S,\varphi=t_xt_y, \mathcal{F})$ are bounded (as in Definition \ref{def:residual-periodicity}) if and only if there exists a positive integer $N$ such that the polynomial $\Theta_N(t) = \Phi_1(t)\cdots\Phi_N(t)$ (see Notation \ref{notation:dynatomic-polynomials}) has a root modulo all but finitely many primes $p$.
\end{prop}

\begin{proof}
Let $\pi:S\rightarrow\mathbb{P}^1$ be the linear fibration through the pair $x,y$.
The polynomial $\Theta_N$ has roots modulo all but finitely many primes if and only if there exist either periodic fibers of $\pi$ of period at most $N$ or fixed points for all but finitely primes. The set of $\mathcal{Z}_\infty(\varphi_p)$ is bounded on periodic curves of period at most $N$ by a bound depending only on $N$, as we have seen in the proof of Proposition \ref{prop:srp-on-cubic-surface-cond1}. Thus for primes $p$ such that $\kappa_p$ is large enough, the fiber must contain points outside $\mathcal{Z}_\infty(\varphi_p)$. Therefore, the condition that there exist either periodic fibers of $\pi$ of period at most $N$ or fixed points for all but finitely many primes, is equivalent to the boundedness of the minimal periods of the residual systems $D_p$.
\end{proof}

\begin{prop} \label{prop:srp-on-cubic-surface-cond2}
Let $S$ be a smooth cubic surface defined over a number field $K$, and let $x,y\in{S(K)}$ be a good pair.
If there exists a positive integer $N$ such that the polynomial $\Theta_N(t) = \Phi_1(t)\cdots\Phi_N(t)$ divided by all linear factors defined over $K$ (i.e.\ all $K$-roots are removed), has no roots over $K$, and for all but finitely many primes $p$ the polynomial $\Theta_N(t)$ has a root modulo $p$, then $D=(S, \varphi=t_xt_y, F)$ is strongly residually periodic.
\end{prop}

\begin{proof}
By Proposition~\ref{prop:cubic-theta-modulo}, we know that the minimal residual periods are bounded by $N$ for all but finitely many primes. We need to check that our points of minimal period modulo $p$ are not all reductions of points from the forbidden set.

First we check reductions of $K$-periodic points. We know that the minimal period for all but finitely many primes is at most $N$, and any periodic point over $K$ must either be of larger period than $N$, or belong to a fiber corresponding to a $K$-root of $\Theta_N(t)$. Let $C$ be a $K$-periodic fiber of period $m>N$, then $y$ is of order $m$ on $C$. For all but finitely many primes, the reduction of the point $y$ modulo $p$ will have the same order $m>N$ on the reduced fiber (cf.\ Silverman \cite[Proposition~VII.3.1]{book:silverman2009}). Therefore the reduction of such a fiber cannot give us points of period at most $N$ (by Proposition \ref{prop:periodic-point-on-fiber-txty}). Now suppose we have a fiber $C$ of the linear fibration that corresponds to a root $a\in{K}$ of $\Theta_N(t)$. Denote by $\tilde{\Theta}_N(t)\in{K[t]}$ the polynomial obtained from $\Theta_N(t)$ after dividing by all linear factors defined over $K$. Then $a$ is not a root of $\tilde{\Theta}_N(t)$. Further, $\rho_p(a)$ can only be a root of $\tilde{\Theta}_N(t)$ modulo $p$ for a finite number of primes $p$. Therefore the fiber $C_p$ can only agree with the fibers lying over the roots of $\tilde{\Theta}_N(t)$ modulo $p$ for finitely many primes $p$. Thus, periodic points obtained from roots of $\tilde{\Theta}_N(t)$ will not be reductions modulo $p$ of points from the fiber $C$.

Finally, we check the reduction of points in $\mathcal{Z}_\infty(\varphi)$: Given a $\kappa_p$ periodic fiber of period at most $N$, it can contain only a bounded amount of periodic points whose orbit goes through $\mathcal{Z}(\varphi_p)$. The bound for this type of points does not depend on $p$ but only on $N$. Therefore, for primes $p$ such that $\kappa_p$ is large enough, there will be $\kappa_p$-periodic points on the fiber that are not reductions of points from $\mathcal{Z}_\infty(\varphi)$ (the orbit of the reduction of such a point must go through $\mathcal{Z}(\varphi_p)$).
\end{proof}

\begin{cor}
Let $S$ be a smooth cubic surface defined over a number field $K$ and let $x,y\in{S(K)}$ be a good pair.
If there exists a number $N$ such that the polynomial $\Theta_N(t) = \Phi_1(t)\cdots\Phi_N(t)$ has no roots over $K$, and for all but finitely many primes $p$ the polynomial $\Theta_N(t)$ has a root modulo $p$, then $D=(S, \varphi=t_xt_y, F)$ is strongly residually periodic.
\end{cor}

\begin{thm} \label{thm:main-theorem-B}
Let $S$ be a smooth cubic surface defined over a number field $K$ and let $x,y\in{S(K)}$ be a good pair.
The system $D=(X,\varphi=t_xt_y,F)$ is strongly residually periodic if and only if one of the following is true:
\begin{enumerate}[(a)]
\item There exists a $K$-fiber in the linear fibration through $x,y$ with finitely many rational points.
\item There exists a positive integer $N$ such that $\Theta_N = \Phi_1 \cdots \Phi_N$ divided by all linear factors over $K$, has roots modulo all but finitely many primes $p$.
\end{enumerate}
\end{thm}

\begin{proof}
The if part is true by Propositions ~\ref{prop:srp-on-cubic-surface-cond1} and \ref{prop:srp-on-cubic-surface-cond2}. If $D$ is strongly residually periodic, then the periods are bounded, and this means there exists a positive integer $N$ such that $\Theta_N$ has roots modulo all but finitely many primes (by Proposition~\ref{prop:cubic-theta-modulo}). If we divide $\Theta_N$ by all the linear factors over $K$, and we have roots modulo all but finitely many primes, then the second condition is satisfied. Otherwise strong residual periodicity must be explained by a fiber that is defined over $K$, but if such a fiber has infinitely many periodic points, then the entire fiber is in the forbidden set -- therefore if the second condition is not met the first condition must be true.
\end{proof}

\section{Finite generation of a subset of rational points on a smooth cubic surface} \label{section:mordell-weil}

In this section we recall the Mordell--Weil problem for cubic surfaces, and prove a dynamical theorem in a Mordell--Weil flavor, using results we obtained in the previous sections. 

Let $S$ be a smooth cubic surface defined over a field $k$. Given a subset of points $\mathcal{S}_0\subseteq{S(k)}$ we can define recursively an increasing sequence of sets 
\begin{equation} \label{eqn:tangent-secant-sequence}
\mathcal{S}_0 \subseteq \mathcal{S}_1 \subseteq \mathcal{S}_2 \subseteq \ldots,
\end{equation}
where $\mathcal{S}_{i+1}$ is defined by adding to $\mathcal{S}_i$ all the points of the form $w\circ{z}\in S(k)$ obtained by taking any two distinct points $w,z\in\mathcal{S}_i$. This is called \emph{drawing secants} through the points in $\mathcal{S}_i$. We can modify this sequence by also including in $\mathcal{S}_{i+1}$ all the points $w\in{S(k)}$ such that $w\in{C_x(k)}$ (recall that $C_x=T_x(S)\cap{S}$) for some $x\in\mathcal{S}_{i}$. This is called \emph{drawing tangents} through the points in $\mathcal{S}_i$. The \emph{span} of $\mathcal{S}_0$ (denoted by $Span(\mathcal{S}_0$)) by tangents and secants (or only secants) is defined as the union of all the sets in the sequence in \textit{(\ref{eqn:tangent-secant-sequence})}. 

\begin{qstn}[\emph{The Mordell--Weil problem for cubic surfaces}]
Given a smooth cubic surface $S$ over a field $k$ such that $S(k)\neq\emptyset$, does there exist a finite subset $\mathcal{S}\subset\mathcal{S}(k)$ such that $Span(\mathcal{S}) = S(k)$? If so, then we say $S(k)$ is \emph{finitely generated}.
\end{qstn}
The Mordell--Weil problem for cubic surfaces is still open, and partial results can be found in Manin \cite{article:manin1997} and Siksek \cite{article:siksek2012}. Of course, the answer to the question may very well depend on whether we allow tangents or not, and in fact Manin considers an alternative composition rule in the above-mentioned article. 

We proceed to prove a Mordell--Weil like theorem for the set of $K$-periodic points of $t_xt_y$ where $K$ is a number field. In order to obtain ``true" finite generation, we restrict ourselves to taking tangents only inside the nonsingular fibers of the linear fibration associated to $x$ and $y$.

\begin{thm} \label{thm:cubic-finitely-generated}
Let $K$ be a number field, $S/K$ a smooth cubic surface and $x,y\in{S(K)}$ a good pair. If there do not exist singular fibers in the linear fibration through $x,y$ that are periodic of finite period, then the set of $K$-periodic points of the birational automorphism $t_xt_y$ is finitely generated by secants and tangents.
\end{thm}

\begin{proof}
By Corollary~\ref{cor:cubic-surface-finite-union-of-fibers-contains-periodic}, there are only finitely many fibers of the linear fibration containing all periodic points. There are two types of such fibers: the first is a singular fiber containing only one singular point, which is fixed under $t_xt_y$ (the rest of the points on the fiber are non-periodic due to the assumption in the theorem). The second is a smooth fiber that is periodic. The first type is surely finitely generated, and the second is finitely generated by the Mordell--Weil theorem for elliptic curves (see Silverman \cite[Chapter VIII]{book:silverman2009}). Thus we can choose a finite number of generators on each periodic fiber, take the union with all the singular points on the singular fibers, and we are done.
\end{proof}

\begin{cor}
If $K=\mathbb{Q}$ in Theorem \ref{thm:cubic-finitely-generated}, and the resultants $Res(\Psi_1, \Psi_n)\neq{0}$ for $n=2,3,4,6$, then the set of $\mathbb{Q}$-periodic points of $t_xt_y$ is finitely generated.
\end{cor}

\begin{proof} The condition 
$Res(\Psi_1, \Psi_n)\neq{0}$ means there are no fibers that are both singular and of period $n$ (by Proposition~\ref{prop:psi1-psi2}). If a singular fiber $C$ of the linear fibration is periodic, then $y$ is of finite order in the group $(C_{ns}, x)$. Singular cubics over $\mathbb{Q}$ have group structure on $C_{ns}$ isomorphic to either $\mathbb{Q}^+$ or a subgroup of $L^*$, where $L$ is a quadratic extension of $\mathbb{Q}$ (see  Silverman \cite[Exercise III.3.5]{book:silverman2009}). In either case there are no non-trivial torsion elements of an order not in $\{2,3,4,6\}$.
\end{proof}

\begin{rem} We show in Example \ref{ex:cubic-singular-periodic-fiber} that there exist dynamical systems on smooth cubic surfaces where $Res(\Psi_1, \Psi_2)=0$. Of course, this does not mean that the consequence of the corollary is untrue for such examples (i.e.\ the set of periodic points might still be finitely generated).
\end{rem}

\section{Examples}

In this section we present some examples, illustrating the ideas and theorems presented in the paper. Finding the division polynomials associated with a good pair $x,y$ is unreasonable without the use of a computer. We have implemented the necessary \magma~\cite{article:magma} functions to calculate the division polynomials associated with the birational map $t_xt_y$. The code package also contains functions that check the minimal residual periods of $t_xt_y$ for primes up to a given bound (this can provide an indication as to whether a given example is SRP or not). The code can be downloaded at the author's website \url{http://www.wishcow.com}. The results of our package can be verified using the \magma\ code package provided by Brown and Ryder \cite{article:brown-ryder2010}, which can generate Geiser involutions for a smooth cubic surface. We start with an example of a strongly residually periodic dynamical system on a smooth cubic surface:

\begin{exmp} \label{ex:srp-cubic} Let $S/\mathbb{Q}$ be the smooth cubic surface defined by the equation
\begin{equation} \label{equation:srp-cubic-equation}
S: YW^2 + Y^2Z-X^3-4Z^3 = 0.
\end{equation}
The points $x=[0:1:0:0], y=[0:-2:1:0]$ are a good pair on $S$.
We denote by $\varphi$ the birational automorphism $t_xt_y$ on $S$.
The forbidden set is defined as in Section \ref{section:SRP-cubic-surface}. We will show that the dynamical system $D=(S/\mathbb{Q}, \varphi, F)$ is $SRP(3)$, i.e.\ there exists a number $M$ such that for any prime $p>M$, the dynamical system has periodic points of period $3$ modulo $p$. 

Intersecting the hyperplane $\{W=0\}$ with $S$ gives the cubic curve $C: Y^2Z-X^3-4Z^3,$ which is a smooth cubic. This hyperplane goes through the points $x$ and $y$, and therefore the curve $C$ is invariant under the map $t_xt_y$ (see Proposition~\ref{prop:cubic-plane-section-invariant}). Choosing the identity of the elliptic group structure to be $x = [0:1:0]$ (or more accurately, the point induced by $x$), we get an elliptic curve with Weierstrass form $v^2=u^3+4.$ Consulting the Cremona database of elliptic curves \cite{article:cremona2006}, we see that $E=(C, x)$ is an elliptic curve with Mordell--Weil rank $0$, and the torsion subgroup is isomorphic to $\mathbb{Z}/3\mathbb{Z}$. The three rational points are $x=[0:1:0], y=[0:-2:1]$ and $z=[0:2:1]$. Since the group is $\mathbb{Z}/3\mathbb{Z}$, the point $y$ is of order $3$ on the elliptic curve. We have already seen (Proposition \ref{prop:periodic-point-on-fiber-txty}) that the order of $y$ is the same as the order of $t_xt_y$ in the group $Bir(E)$. Thus we get that even though $t_xt_y$ is of infinite order in $Bir(S)$, it is of finite order $3$ when restricted to the hyperplane section. By Proposition \ref{prop:srp-on-cubic-surface-cond1} we get that the dynamical system $D$ is SRP(3).
\end{exmp}

Let us describe a method for constructing examples with desired dynamical properties (e.g.\ SRP). We choose an elliptic curve $E: v^2=u^3+Au+B$ with particular properties, for instance with a finite number of rational points as was done in Example \ref{ex:srp-cubic}. We then choose two points $x,y\in{E(K)}$ to play the roles of the points $x,y$ on the surface $S$, and search for a smooth surface $S$ containing this curve. We can do this by running over surfaces of the form $$S: W\cdot{H(X,Y,Z,W)}-Y^2Z-X^3-AXZ^2-BZ^3=0$$ where $H$ is a quadratic form, and searching for ones that are smooth and in which $x,y$ is a good pair. Then this surface will have $C=\{W=0\}\cap{S}$ as a fiber in the linear fibration induced by $x,y$, and this fiber $C$ is exactly our curve $E$.

\begin{exmp} \label{ex:srp-cubic-contd}
We continue Example \ref{ex:srp-cubic}, by proving that the dynamical system $D$ has no periodic points over $\mathbb{Q}$. We use this example to illustrate the method of using the division polynomials of the linear fibration induced by the good pair $x,y$. To find the cubic pencil of the linear fibration of $S$ induced by $x,y$, we set $W=tX$ in equation (\ref{equation:srp-cubic-equation}). This is because the line passing through $x,y$ is $\mathcal{L}(x,y)=\{W=0,X=0\}$. We get the cubic pencil $C: t^2X^2Y+Y^2Z-X^3-4Z^3=0.$

The fiber at infinity that we have removed is $X=0$, which is the curve $C_\infty :YW^2 + Y^2Z-4Z^3=0.$ We bring this curve to Weierstrass form and get $E_\infty: v^2 = u^3 - 4u.$ This curve can be checked in the Cremona database \cite{article:cremona2006} to have Mordell--Weil rank $0$, and torsion subgroup $\mathbb{Z}/2\mathbb{Z} \times \mathbb{Z}/2\mathbb{Z}$. All four rational points on $C_\infty$ can be checked to be in $\mathcal{Z}_\infty(t_xt_y)$ (and therefore are not periodic and are in the forbidden set). The point $y$ has order $2$ in the group $E_\infty(\mathbb{Q})$, so that by Proposition \ref{prop:srp-on-cubic-surface-cond1} the dynamical system $D$ is in fact $SRP(2)$.

We can check that the Weierstrass form of the cubic pencil $C$ is $E: v^2=u^3-4t^4u+4.$ 
The discriminant of $E$ is 
$\Delta(t) = 4096t^{12}-6912.$
The zeros of the discriminant, as a polynomial in $t$, give us $12$ distinct singular fibers of the linear fibration, all of which are nodal cubics (the coefficient of $u$ is nonzero at all roots of $\Delta(t)$, see Silverman \cite[Proposition~III.1.4]{book:silverman2009}). None of the roots of $\Delta(t)$ are rational, so there are no fixed points of $t_xt_y$ defined over $\mathbb{Q}$ (See Corollary \ref{cor:txty-fixed-iff-singular}). 

The dynamical system is $SRP(1)$ only if the polynomial $\Delta$ has roots modulo all but finitely many primes $p$. We show this is not true. We can check using \magma~\cite{article:magma} that the Galois group of the polynomial $\Delta$ has an element of order $12$. Then by the Frobenius density theorem (see Lenstra and Stevenhagen \cite[page 32]{article:lenstra-stevenhagen1996}), there are infinitely many primes $p$ for which the polynomial $\Delta$ remains irreducible when reduced modulo $p$, so that residual periodicity cannot be explained by the fixed points.

To get a fiber of the linear fibration that is periodic of period $2$, we need the $v$ coordinate of the point $y$ to be $0$ when evaluated at $t$. However, the image of $y$ in $E$ is $[0:2:1]$, so we get that outside of the fiber at $\infty$ there are no fibers of period $2$.

We can calculate the division polynomials $\Psi_n, n\geq{3}$ for $t_xt_y$. To find them we take the division polynomials $\psi_n$ for $E$, and evaluate them at $v=0$, since $y=[0:2:1]$. We only need the polynomials $\Psi_n(t)$ for $n\leq{12}, n\neq{11}$ (See Proposition \ref{prop:bounded-number-of-periodic-fibers} for why we can skip $\Psi_{11}$, and stop at $\Psi_{12}$). The polynomials are quite large so we do not list them here.
A quick check shows that none of these polynomials have roots in $\mathbb{Q}$ outside of $t=0$, which is the fiber of period $3$. This means that there are no fibers of finite period $n\geq{3}$.  As we have checked all possibilities, we have proved that $t_xt_y$ has no periodic points over $\mathbb{Q}$.
\end{exmp}

\begin{exmp}\label{ex:cubic-infinite-rational-periodic-points}
We show an example of a dynamical system on a smooth cubic surface that has infinitely many periodic points.
Let $S/\mathbb{Q}$ be the smooth cubic surface defined by the equation
\begin{equation}
S: X^3 - 3024XZ^2 - Y^2Z - YW^2 + 81216Z^3 = 0.
\end{equation}
The points $x=[0,1,0,0], y=[12,216,1,0]$ form a good pair on $S$. The hyperplane section $C = \{W=0\}\cap{S}$ is a fiber of the linear fibration induced by $x,y$. This cubic curve has the Weierstrass form $E: v^2=u^3-3024u+81216.$
This curve can be checked in the Cremona database \cite{article:cremona2006} to have Mordell--Weil rank $1$, and $y=(12,216)$ has order $3$ in the elliptic curve $E$. This means that the fiber $C$ is periodic of period $3$ under $t_xt_y$, which proves there are infinitely many $\mathbb{Q}$-periodic points for $t_xt_y$ on $S$. 
\end{exmp}

\begin{exmp} \label{ex:cubic-singular-periodic-fiber}
We show an example of a dynamical system on a smooth cubic surface with a singular fiber in the linear fibration containing infinitely many periodic points. This example demonstrates that the condition in Theorem~\ref{thm:cubic-finitely-generated} is not redundant (in the sense that such systems exist, not that the condition is necessary).
Let $S/\mathbb{Q}$ be the smooth cubic surface defined by the equation
\begin{equation}
S: X^3 + X^2Z - XYW - Y^2Z - Z^2W - W^3 = 0.
\end{equation}
The points $x=[0,1,0,0], y=[-1,0,1,0]$ form a good pair on $S$. The hyperplane section $C = \{W=0\}\cap{S}$ is a fiber of the linear fibration induced by $x,y$. This cubic curve has the Weierstrass form $E: v^2=u^3+u^2,$ which is the classic nodal cubic. The point $y=(-1,0)$ on $E$ is the unique point of order $2$ of this curve. This means that $S$ has a singular curve of period $2$ under $t_xt_y$. 
\end{exmp}


\begin{thebibliography}{10}

\bibitem{article:amerik2011}
Ekaterina Amerik, \emph{Existence of non-preperiodic algebraic points for a
  rational self-map of infinite order}, Math. Res. Lett. \textbf{18} (2011),
  no.~2, 251--256. \MR{2784670 (2012e:37182)}

\bibitem{article:bandman-grunewald-kunyavskii2010}
Tatiana Bandman, Fritz Grunewald, and Boris Kunyavski{\u\i}, \emph{Geometry and
  arithmetic of verbal dynamical systems on simple groups}, Groups Geom. Dyn.
  \textbf{4} (2010), no.~4, 607--655, With an appendix by Nathan Jones.
  \MR{2727656 (2011k:14020)}

\bibitem{book:beltrametti-et-al2009}
Mauro~C. Beltrametti, Ettore Carletti, Dionisio Gallarati, and Giacomo
  Monti~Bragadin, \emph{Lectures on curves, surfaces and projective varieties},
  EMS Textbooks in Mathematics, European Mathematical Society (EMS), Z\"urich,
  2009, A classical view of algebraic geometry, Translated from the 2003
  Italian original by Francis Sullivan. \MR{2549804 (2010k:14001)}

\bibitem{article:blanc-cantat2013}
J\`{e}r\`{e}my Blanc and Serge Cantat, \emph{Dynamical degrees of birational
  transformations of projective surfaces},  (2013), Preprint.

\bibitem{article:magma}
Wieb Bosma, John Cannon, and Catherine Playoust, \emph{The {M}agma algebra
  system. {I}. {T}he user language}, J. Symbolic Comput. \textbf{24} (1997),
  no.~3-4, 235--265, Computational algebra and number theory (London, 1993).
  \MR{1484478}

\bibitem{article:brown-ryder2010}
Gavin Brown and Daniel Ryder, \emph{Elliptic fibrations on cubic surfaces}, J.
  Pure Appl. Algebra \textbf{214} (2010), no.~4, 410--421. \MR{2558749
  (2011a:14023)}

\bibitem{article:corti-pukhlikov-reid2000}
Alessio Corti, Aleksandr Pukhlikov, and Miles Reid, \emph{Fano {$3$}-fold
  hypersurfaces}, Explicit birational geometry of 3-folds, London Math. Soc.
  Lecture Note Ser., vol. 281, Cambridge Univ. Press, Cambridge, 2000,
  pp.~175--258. \MR{1798983 (2001k:14034)}

\bibitem{book:cremona1997}
J.~E. Cremona, \emph{Algorithms for modular elliptic curves}, second ed.,
  Cambridge University Press, Cambridge, 1997. \MR{1628193 (99e:11068)}

\bibitem{article:cremona2006}
John Cremona, \emph{The elliptic curve database for conductors to 130000},
  Algorithmic number theory, Lecture Notes in Comput. Sci., vol. 4076,
  Springer, Berlin, 2006, pp.~11--29. \MR{2282912 (2007k:11087)}

\bibitem{article:diller1996}
Jeffrey Diller, \emph{Dynamics of birational maps of {${\bf P}^2$}}, Indiana
  Univ. Math. J. \textbf{45} (1996), no.~3, 721--772. \MR{1422105 (97k:32044)}

\bibitem{book:hindry-silverman2000}
Marc Hindry and Joseph~H. Silverman, \emph{Diophantine geometry}, Graduate
  Texts in Mathematics, vol. 201, Springer-Verlag, New York, 2000, An
  introduction. \MR{1745599 (2001e:11058)}

\bibitem{article:hutz2009}
Benjamin Hutz, \emph{Good reduction of periodic points on projective
  varieties}, Illinois J. Math. \textbf{53} (2009), no.~4, 1109--1126.
  \MR{2741181 (2012e:37184)}

\bibitem{article:hutz2012}
Benjamin Hutz, \emph{{Determination of all rational preperiodic points for
  morphisms of $\mathbb{P}^N$}}, preprint (2012),
  \href{http://arxiv.org/abs/1210.6246}{arXiv:1210.6246}.

\bibitem{book:manin1986}
Yu.~I. Manin, \emph{Cubic forms}, second ed., North-Holland Mathematical
  Library, vol.~4, North-Holland Publishing Co., Amsterdam, 1986, Algebra,
  geometry, arithmetic, Translated from the Russian by M. Hazewinkel.
  \MR{833513 (87d:11037)}

\bibitem{article:manin1997}
\bysame, \emph{Mordell-{W}eil problem for cubic surfaces}, Advances in
  mathematical sciences: {CRM}'s 25 years ({M}ontreal, {PQ}, 1994), CRM Proc.
  Lecture Notes, vol.~11, Amer. Math. Soc., Providence, RI, 1997, pp.~313--318.
  \MR{1479681 (99a:14029)}

\bibitem{article:mazur1978}
B.~Mazur, \emph{Modular curves and the {E}isenstein ideal}, Inst. Hautes
  \'Etudes Sci. Publ. Math. (1977), no.~47, 33--186 (1978). \MR{488287
  (80c:14015)}

\bibitem{article:merel1996}
Lo{\"{\i}}c Merel, \emph{Bornes pour la torsion des courbes elliptiques sur les
  corps de nombres}, Invent. Math. \textbf{124} (1996), no.~1-3, 437--449.
  \MR{1369424 (96i:11057)}

\bibitem{article:miranda1990}
Rick Miranda, \emph{Persson's list of singular fibers for a rational elliptic
  surface}, Math. Z. \textbf{205} (1990), no.~2, 191--211. \MR{1076128
  (92a:14035)}

\bibitem{article:parent1999}
Pierre Parent, \emph{Bornes effectives pour la torsion des courbes elliptiques
  sur les corps de nombres}, J. Reine Angew. Math. \textbf{506} (1999),
  85--116. \MR{1665681 (99k:11080)}

\bibitem{book:reid1988}
Miles Reid, \emph{Undergraduate algebraic geometry}, London Mathematical
  Society Student Texts, vol.~12, Cambridge University Press, Cambridge, 1988.
  \MR{982494 (90a:14001)}

\bibitem{book:shafarevich1994-vol1}
Igor~R. Shafarevich, \emph{Basic algebraic geometry. 1}, second ed.,
  Springer-Verlag, Berlin, 1994, Varieties in projective space, Translated from
  the 1988 Russian edition and with notes by Miles Reid. \MR{1328833
  (95m:14001)}

\bibitem{article:shioda1995}
Tetsuji Shioda, \emph{Weierstrass transformations and cubic surfaces}, Comment.
  Math. Univ. St. Paul. \textbf{44} (1995), no.~1, 109--128. \MR{1336422
  (96h:14046)}

\bibitem{article:siksek2012}
Samir Siksek, \emph{On the number of {M}ordell-{W}eil generators for cubic
  surfaces}, J. Number Theory \textbf{132} (2012), no.~11, 2610--2629.
  \MR{2954995}

\bibitem{article:silverman2008}
Joseph~H. Silverman, \emph{Variation of periods modulo {$p$} in arithmetic
  dynamics}, New York J. Math. \textbf{14} (2008), 601--616. \MR{2448661
  (2009h:11100)}

\bibitem{book:silverman2009}
\bysame, \emph{The arithmetic of elliptic curves}, second ed., Graduate Texts
  in Mathematics, vol. 106, Springer, Dordrecht, 2009. \MR{2514094
  (2010i:11005)}

\bibitem{article:lenstra-stevenhagen1996}
P.~Stevenhagen and H.~W. Lenstra, Jr., \emph{Chebotar\"ev and his density
  theorem}, Math. Intelligencer \textbf{18} (1996), no.~2, 26--37. \MR{1395088
  (97e:11144)}

\bibitem{book:washington2008}
Lawrence~C. Washington, \emph{Elliptic curves}, second ed., Discrete
  Mathematics and its Applications (Boca Raton), Chapman \& Hall/CRC, Boca
  Raton, FL, 2008, Number theory and cryptography. \MR{2404461 (2009b:11101)}

\end{thebibliography}

\def\cprime{$'$}
\providecommand{\bysame}{\leavevmode\hbox to3em{\hrulefill}\thinspace}
\providecommand{\MR}{\relax\ifhmode\unskip\space\fi MR }
\providecommand{\MRhref}[2]{%
  \href{http://www.ams.org/mathscinet-getitem?mr=#1}{#2}
}
\providecommand{\href}[2]{#2}

\end{document}